\documentclass[11pt,reqno]{amsart}

\usepackage{amsmath,amssymb,amsfonts,mathtools}
\usepackage{enumitem}
\usepackage{microtype}
\usepackage{hyperref}

\newcommand{\Span}{\operatorname{span}}
\newcommand{\spec}{\operatorname{spec}}
\newcommand{\diag}{\operatorname{diag}}
\newcommand{\norm}[1]{\left\|#1\right\|}
\newcommand{\ip}[2]{\left\langle #1,#2\right\rangle}

\theoremstyle{plain}
\newtheorem{theorem}{Theorem}[section]
\newtheorem{proposition}[theorem]{Proposition}
\newtheorem{lemma}[theorem]{Lemma}
\newtheorem{corollary}[theorem]{Corollary}

\theoremstyle{remark}
\newtheorem{remark}[theorem]{Remark}

\title[Elliptical width and the Bourin--Lee conjecture]
{An Extremal Formula for Elliptical Width
	and the Bourin--Lee Conjecture}

\author{Mohammad Sababheh}

\address{Department of Mathematics,
	Abdullah Al Salem University, Kuwait. \newline	
	 Department of Basic Sciences, Princess Sumaya University for Technology, Amman, Jordan}

\email{sababheh@yahoo.com}

\subjclass[2020]{15A60, 47A12, 15A42, 47A30}

\keywords{Positive block matrix, numerical range, elliptical width,
	operator norm, compression, essentially Hermitian matrix,
	Bourin--Lee conjecture}

\begin{document}
	
	\begin{abstract}
		Let \(X\in M_n(\mathbb C)\), \(n\geq2\), and let
		\[
		\delta_2(X)
		=
		\sup_{\dim S=2}\delta(X_S),
		\]
		where \(X_S\) denotes the compression of \(X\) to \(S\), and
		\(\delta(X_S)\) is the diameter of the largest disk contained in
		the numerical range \(W(X_S)\). Bourin and Lee proved that every
		positive block matrix
		\[
		M=
		\begin{pmatrix}
			A&X\\
			X^*&B
		\end{pmatrix}
		\]
		satisfies
		\[
		\norm{M}\leq\norm{A+B}+\delta_2(X).
		\]
		In the same work, they conjectured that the inequality
		\[
		\norm{M}\leq\norm{A+B}
		\]
		for every positive block matrix with prescribed off-diagonal block
		\(X\) characterizes essentially Hermitian matrices.
		
		We resolve this conjecture by proving the stronger exact formula
		\[
		\sup_{\left(\begin{smallmatrix}A&X\\X^*&B\end{smallmatrix}\right)\geq0}
		\left\{
		\norm{\begin{pmatrix}A&X\\X^*&B\end{pmatrix}}
		-\norm{A+B}
		\right\}
		=
		\delta_2(X).
		\]
		Thus the elliptical width is precisely the optimal norm defect
		associated with a fixed off-diagonal block. The reverse inequality
		is obtained by lifting an arbitrary two-dimensional compression
		of \(X\) to an explicit family of positive block matrices, yielding
		quantitative two-sided estimates with error of order \(t^{-1}\).
		
		As a consequence, the norm inequality above holds universally
		if and only if \(X\) is essentially Hermitian, proving Conjecture 3.3
		of Bourin and Lee. We also construct an explicit \(3\times3\)
		normal example showing that the radius constant in their
		normal-off-diagonal eigenvalue estimate is optimal at the leading
		eigenvalue level. By direct-sum amplification of this construction,
		we further prove that the same radius constant is sharp at every
		eigenvalue level \(j\geq0\), thereby resolving their corresponding
		sharpness question.
	\end{abstract}
	
	\maketitle
	
	\section{Introduction}
	
	Let
	\[
	M=
	\begin{pmatrix}
		A&X\\
		X^*&B
	\end{pmatrix}\geq0,
	\qquad
	A,B,X\in M_n(\mathbb C).
	\]
	Positive block matrices of this form occur throughout matrix
	analysis and operator theory.  A basic theme is to compare spectral
	or norm quantities associated with the full matrix \(M\) with
	corresponding quantities associated with the partial trace \(A+B\).
	The off-diagonal block \(X\) represents the obstruction to such
	comparisons, and numerical-range geometry provides a natural way of
	measuring this obstruction; see, for example,
	\cite{BourinLeeNumerical,BourinLeeInradius,BourinMhanna,
		GumusLiuRaouafiTam,Hayashi}.
	
	Throughout the paper,
	\[
	\ip{x}{y}=x^*y,
	\]
	so the inner product is linear in its second variable.
	
	For a compact convex set \(\Lambda\subset\mathbb C\), let
	\(\delta(\Lambda)\) denote the diameter of the largest closed disk
	contained in \(\Lambda\).  For \(T\in M_n(\mathbb C)\), write
	\[
	\delta(T)=\delta(W(T)),
	\]
	where
	\[
	W(T)=
	\left\{
	\ip{x}{Tx}:\norm{x}=1
	\right\}
	\]
	is the numerical range of \(T\).
	
	If \(S\subset\mathbb C^n\) is a subspace and \(P_S\) denotes the
	orthogonal projection onto \(S\), then the compression of
	\(T\in M_n(\mathbb C)\) to \(S\) is the operator
	\[
	T_S=P_ST|_S:S\to S.
	\]
	Thus, for \(x\in S\),
	\[
	T_Sx=P_STx.
	\]

Bourin and Lee \cite{BourinLeeInradius} introduced the
\emph{elliptical width} of \(X\in M_n(\mathbb C)\) by
\begin{equation}\label{eq:delta2}
	\delta_2(X)
	=
	\sup_{\dim S=2}\delta(X_S).
\end{equation}
The terminology is motivated by the elliptical range theorem:
for every two-dimensional subspace \(S\subset\mathbb C^n\),
the numerical range \(W(X_S)\) is an elliptical disk, possibly
degenerate, and \(\delta(X_S)\) is exactly the length of its minor
axis.
	
	A principal result of Bourin and Lee
	\cite[Theorem 2.1]{BourinLeeInradius} asserts that, whenever
	\[
	M=
	\begin{pmatrix}
		A&X\\
		X^*&B
	\end{pmatrix}
	\geq0,
	\]
	one has
	\begin{equation}\label{eq:BL-eigenvalue}
		\lambda_{1+2j}(M)
		\leq
		\lambda_{1+j}(A+B)+\delta_2(X),
		\qquad
		j=0,\ldots,n-1.
	\end{equation}
	Here, for a Hermitian matrix \(H\), its eigenvalues are arranged in
	decreasing order,
	\[
	\lambda_1(H)\geq\lambda_2(H)\geq\cdots.
	\]
	
At the operator-norm level, the case \(j=0\) of
\eqref{eq:BL-eigenvalue} yields
\begin{equation}\label{eq:BL-norm}
	\norm{
		\begin{pmatrix}
			A&X\\
			X^*&B
	\end{pmatrix}}
	\leq
	\norm{A+B}+\delta_2(X).
\end{equation}
	
The same work raises a natural rigidity question concerning the case
in which the correction term in \eqref{eq:BL-norm} can be omitted.
Recall that a matrix \(X\in M_n(\mathbb C)\) is called
\emph{essentially Hermitian} if
\[
X=\alpha H+\beta I
\]
for some Hermitian matrix \(H\) and some
\(\alpha,\beta\in\mathbb C\).  In this case, for every subspace
\(S\subset\mathbb C^n\), the compression
\[
X_S=\alpha H_S+\beta I_S
\]
is normal and its numerical range is a line segment.  Hence
\[
\delta_2(X)=0,
\]
and \eqref{eq:BL-norm} reduces to
\begin{equation}\label{eq:zero-inequality}
	\norm{
		\begin{pmatrix}
			A&X\\
			X^*&B
	\end{pmatrix}}
	\leq
	\norm{A+B}.
\end{equation}
	
Motivated in particular by a theorem of Hayashi
\cite{Hayashi}, Bourin and Lee formulated the following conjecture
\cite[Conjecture 3.3]{BourinLeeInradius}:

\medskip

\noindent
\emph{If \eqref{eq:zero-inequality} holds for every positive block
	matrix having \(X\) as its off-diagonal block, then \(X\) is
	essentially Hermitian.}

\medskip

To the best of our knowledge, this conjecture has remained open.
The principal aim of the present paper is to resolve it and, in fact,
to place it within a stronger quantitative framework.
	
Our approach yields a stronger quantitative statement.  For
\(X\in M_n(\mathbb C)\), define the positive-block norm defect by
\begin{equation}\label{eq:Phi}
	\Phi(X)
	=
	\sup
	\left\{
	\norm{
		\begin{pmatrix}
			A&X\\
			X^*&B
	\end{pmatrix}}
	-\norm{A+B}:
	\begin{pmatrix}
		A&X\\
		X^*&B
	\end{pmatrix}\geq0
	\right\}.
\end{equation}
The admissible family is nonempty.  Indeed, if
\(c\geq\norm{X}\), then
\[
\begin{pmatrix}
	cI&X\\
	X^*&cI
\end{pmatrix}\geq0.
\]

Our main theorem gives an exact formula for this quantity:
\begin{equation}\label{eq:intro-main}
	\Phi(X)=\delta_2(X).
\end{equation}
Consequently, the elliptical width appearing in
\eqref{eq:BL-norm} is not merely a valid additive correction term;
it is the optimal correction associated with the prescribed
off-diagonal block \(X\).
	
The estimate
\[
\Phi(X)\leq\delta_2(X)
\]
is precisely the \(j=0\) case of \eqref{eq:BL-eigenvalue}.  The new
content is the reverse inequality.  More precisely, for every
two-dimensional subspace \(S\subset\mathbb C^n\), we construct a
family of positive block matrices
\[
M_t=
\begin{pmatrix}
	A_t&X\\
	X^*&B_t
\end{pmatrix}\geq0
\]
such that
\[
\norm{M_t}-\norm{A_t+B_t}
\longrightarrow
\delta(X_S)
\qquad (t\to\infty).
\]
Thus every two-dimensional compression of \(X\) can be realized
asymptotically through the norm defect of positive block matrices
having the same prescribed off-diagonal block \(X\).  Taking the
supremum over all two-dimensional subspaces \(S\) yields the reverse
inequality
\[
\Phi(X)\geq\delta_2(X).
\]
	
Fix a two-dimensional subspace \(S\subset\mathbb C^n\).  The
construction is based on a Schur form of the compression \(X_S\).
After choosing a suitable orthonormal basis of \(S\), we may write
\[
X_S=
\begin{pmatrix}
	\alpha&z\\
	0&\beta
\end{pmatrix}.
\]
By the elliptical range theorem; see, for example,
\cite{Li,HornJohnson}, the numerical range \(W(X_S)\) is an
elliptical disk, possibly degenerate, with foci \(\alpha\) and
\(\beta\).  Its minor-axis length is
\[
\left(
\operatorname{tr}(X_S^*X_S)-|\alpha|^2-|\beta|^2
\right)^{1/2}
=
|z|.
\]
Therefore
\[
\delta(X_S)=|z|.
\]

The construction proceeds by introducing opposite large Hermitian
perturbations along the two Schur directions.  The upper spectral
edge then captures the contribution of the off-diagonal entry \(z\),
whereas the lower spectral edge is governed by the vanishing reverse
entry in the Schur form.  After shifting by the lower spectral edge,
one obtains a positive block matrix, and elementary quadratic-form
estimates yield an explicit error of order \(t^{-1}\).  This gives
the desired asymptotic realization of \(\delta(X_S)\).

The exact identity \eqref{eq:intro-main} reduces the Bourin--Lee
conjecture to the zero set of the elliptical width.  We prove, in a
self-contained finite-dimensional argument, that
\begin{equation}\label{eq:intro-zero}
	\delta_2(X)=0
	\quad\Longleftrightarrow\quad
	X\ \text{is essentially Hermitian}.
\end{equation}
This characterization is closely related to classical results on
normal two-dimensional compressions; compare Farenick
\cite{Farenick}.  Combining \eqref{eq:intro-main} with
\eqref{eq:intro-zero} yields the resolution of the Bourin--Lee
conjecture.
	
There is also a second open problem in
\cite{BourinLeeInradius}. If \(N\) is normal and
\[
\operatorname{spec}(N)\subset\{z\in\mathbb C:|z-\gamma|\le r\}
\]
for some \(\gamma\in\mathbb C\), then Bourin and Lee proved
\[
\lambda_{1+2j}
\begin{pmatrix}
	A&N\\
	N^*&B
\end{pmatrix}
\leq
\lambda_{1+j}(A+B)+r,
\qquad
j=0,\ldots,n-1.
\]
They asked whether the constant \(r\) is sharp. We first construct an
explicit \(3\times3\) normal matrix together with an asymptotically
extremal family of positive block matrices, which proves sharpness at
the leading eigenvalue level \(j=0\). A direct-sum amplification of
this construction then shows that, for every prescribed \(j\geq0\),
the constant \(r\) is asymptotically attained at the corresponding
eigenvalue level. Thus the radius constant in the Bourin--Lee estimate
is best possible for every \(j\geq0\).
	
The paper is organized as follows. Section 2 collects the
two-dimensional numerical-range facts and the elementary block
estimates used in the lifting argument. Section 3 establishes the
quantitative compression-lifting theorem and the exact identity
\eqref{eq:intro-main}. Section 4 characterizes the zero set of the
elliptical width and resolves Conjecture 3.3 of Bourin and Lee.
Section 5 treats normal off-diagonal blocks and proves that the radius
constant in the Bourin--Lee eigenvalue estimate is sharp at every
eigenvalue level \(j\geq0\).
%%%%%%%%%%%%%%%%%%%%%%%%%%%%%%%%%%%%%%%%%%%%%%%%%%%%%%%%%%%%%%%%%%%%%%%%%%%%%%%%%%%%%%%%%%%%%%%%%%%%%%%%%%%%%%%%%%%%%%%%%%%%%%%%%%%%%%%%%%%%%%%%%%%%%%%%%%%%%%%%%%%%%%%%%%%%%%%%%%%%%%%%%%%%%%%%%%%%%%%%%%
	
\section{Two-dimensional geometry and block estimates}

We begin with the elementary two-dimensional fact underlying the
definition of \(\delta_2\).

\begin{lemma}\label{lem:ellipse}
	Let \(T\in M_2(\mathbb C)\).  If, with respect to some orthonormal
	basis,
	\[
	T=
	\begin{pmatrix}
		\alpha&z\\
		0&\beta
	\end{pmatrix},
	\]
	then
	\[
	\delta(T)=|z|.
	\]
\end{lemma}

\begin{proof}
	By the elliptical range theorem, \(W(T)\) is an elliptical disk,
	possibly degenerate, with foci \(\alpha\) and \(\beta\); see, for
	example, \cite{HornJohnson,Li}.  Its minor-axis length is
	\[
	\left(
	\operatorname{tr}(T^*T)-|\alpha|^2-|\beta|^2
	\right)^{1/2}.
	\]
	Since
	\[
	\operatorname{tr}(T^*T)
	=
	|\alpha|^2+|\beta|^2+|z|^2,
	\]
	the minor-axis length is \(|z|\).  The largest disk contained in an
	elliptical disk has diameter equal to its minor-axis length.
	Therefore
	\[
	\delta(T)=|z|.
	\]
\end{proof}
	
The next two elementary estimates control the upper and lower
spectral edges of the block matrices used in the lifting argument.

\begin{lemma}\label{lem:block-upper}
	Let \(E\) be a nonzero subspace of \(\mathbb C^m\), so that
	\[
	\mathbb C^m=E\oplus E^\perp.
	\]
	Let
	\[
	Q=
	\begin{pmatrix}
		tI_E+A&C\\
		C^*&R
	\end{pmatrix}
	\]
	be a Hermitian matrix written in block form relative to this
	orthogonal decomposition, where \(A\) and \(R\) are Hermitian
	operators on \(E\) and \(E^\perp\), respectively, and
	\(C:E^\perp\to E\).  Suppose that \(M\geq0\),
	\[
	\lambda_{\max}(A)=a,\qquad
	\norm{C}\leq M,\qquad
	R\leq M I_{E^\perp},
	\]
	and
	\[
	t+a>M.
	\]
	Then
	\begin{equation}\label{eq:block-upper}
		t+a
		\leq
		\lambda_{\max}(Q)
		\leq
		t+a+\frac{M^2}{t+a-M}.
	\end{equation}
\end{lemma}

\begin{proof}
	The compression of \(Q\) to \(E\) is \(tI_E+A\).  Hence, by the
	variational characterization of the largest eigenvalue,
	\[
	\lambda_{\max}(Q)
	\geq
	\lambda_{\max}(tI_E+A)
	=
	t+a.
	\]
	
	For the reverse estimate, let \(x\in E\) and \(y\in E^\perp\).
	Using
	\[
	\ip{x}{Ax}\leq a\norm{x}^2,\qquad
	\bigl|\ip{x}{Cy}\bigr|
	\leq\norm{C}\norm{x}\norm{y}
	\leq M\norm{x}\norm{y},
	\]
	and
	\[
	\ip{y}{Ry}\leq M\norm{y}^2,
	\]
	we obtain
	\[
	\ip{\binom{x}{y}}{Q\binom{x}{y}}
	\leq
	(t+a)\norm{x}^2
	+2M\norm{x}\norm{y}
	+M\norm{y}^2.
	\]
	Set
	\[
	\varepsilon
	=
	\frac{M^2}{t+a-M}.
	\]
	Since \(t+a>M\), we have \(\varepsilon\geq0\), and
	\[
	\varepsilon\bigl(t+a+\varepsilon-M\bigr)-M^2
	=
	\varepsilon^2\geq0.
	\]
	Therefore
	\[
	\begin{pmatrix}
		\varepsilon&-M\\
		-M&t+a+\varepsilon-M
	\end{pmatrix}
	\geq0,
	\]
	and consequently
	\[
	2M\norm{x}\norm{y}
	\leq
	\varepsilon\norm{x}^2
	+
	\bigl(t+a+\varepsilon-M\bigr)\norm{y}^2.
	\]
	Substituting this estimate into the preceding quadratic-form bound
	gives
	\[
	\ip{\binom{x}{y}}{Q\binom{x}{y}}
	\leq
	(t+a+\varepsilon)
	\left(\norm{x}^2+\norm{y}^2\right).
	\]
	Taking the supremum over all unit vectors
	\(\binom{x}{y}\in E\oplus E^\perp\) yields
	\[
	\lambda_{\max}(Q)
	\leq
	t+a+\varepsilon
	=
	t+a+\frac{M^2}{t+a-M}.
	\]
	This proves \eqref{eq:block-upper}.
\end{proof}

\begin{lemma}\label{lem:block-lower}
	Let \(E\) be a nonzero subspace of \(\mathbb C^m\), so that
	\[
	\mathbb C^m=E\oplus E^\perp.
	\]
	Let
	\[
	Q=
	\begin{pmatrix}
		-tI_E&C\\
		C^*&R
	\end{pmatrix}
	\]
	be a Hermitian matrix written in block form relative to this
	orthogonal decomposition, where \(R\) is Hermitian on \(E^\perp\)
	and \(C:E^\perp\to E\).  Suppose that \(M\geq0\),
	\[
	\norm{C}\leq M,\qquad
	R\geq -M I_{E^\perp},
	\]
	and
	\[
	t>M.
	\]
	Then
	\begin{equation}\label{eq:block-lower}
		-t-\frac{M^2}{t-M}
		\leq
		\lambda_{\min}(Q)
		\leq
		-t.
	\end{equation}
\end{lemma}

\begin{proof}
	The compression of \(Q\) to \(E\) is \(-tI_E\).  Hence, by the
	variational characterization of the smallest eigenvalue,
	\[
	\lambda_{\min}(Q)
	\leq
	-t.
	\]
	
	For the reverse estimate, let \(x\in E\) and \(y\in E^\perp\).
	Using
	\[
	2\Re\ip{x}{Cy}
	\geq
	-2\bigl|\ip{x}{Cy}\bigr|
	\geq
	-2M\norm{x}\norm{y}
	\]
	and
	\[
	\ip{y}{Ry}\geq -M\norm{y}^2,
	\]
	we obtain
	\[
	\ip{\binom{x}{y}}{Q\binom{x}{y}}
	\geq
	-t\norm{x}^2
	-2M\norm{x}\norm{y}
	-M\norm{y}^2.
	\]
	
	Set
	\[
	\eta=\frac{M^2}{t-M}.
	\]
	Since \(t>M\), we have \(\eta\geq0\), and
	\[
	\eta(t+\eta-M)-M^2
	=
	\eta^2\geq0.
	\]
	Therefore
	\[
	\begin{pmatrix}
		\eta&-M\\
		-M&t+\eta-M
	\end{pmatrix}
	\geq0,
	\]
	and consequently
	\[
	2M\norm{x}\norm{y}
	\leq
	\eta\norm{x}^2
	+
	(t+\eta-M)\norm{y}^2.
	\]
	Substituting this estimate into the preceding quadratic-form bound
	gives
	\[
	\ip{\binom{x}{y}}{Q\binom{x}{y}}
	\geq
	-(t+\eta)
	\left(\norm{x}^2+\norm{y}^2\right).
	\]
	Taking the infimum over all unit vectors
	\(\binom{x}{y}\in E\oplus E^\perp\) yields
	\[
	\lambda_{\min}(Q)
	\geq
	-t-\eta
	=
	-t-\frac{M^2}{t-M}.
	\]
	This proves \eqref{eq:block-lower}.
\end{proof}
	
\section{The exact norm-defect formula}

We now prove the main lifting result.

\begin{theorem}[Quantitative compression lifting]
	\label{thm:lifting}
	Let \(X\in M_n(\mathbb C)\), let
	\(S\subset\mathbb C^n\) be a two-dimensional subspace, and set
	\[
	m=\norm{X}.
	\]
	For every \(t>m\), there exist positive semidefinite matrices
	\(A_t,B_t\in M_n(\mathbb C)\) such that
	\[
	M_t=
	\begin{pmatrix}
		A_t&X\\
		X^*&B_t
	\end{pmatrix}
	\geq0
	\]
	and
	\begin{equation}\label{eq:quantitative-lifting}
		\delta(X_S)-\frac{m^2}{t-m}
		\leq
		\norm{M_t}-\norm{A_t+B_t}
		\leq
		\delta(X_S)
		+
		\frac{m^2}{t+\delta(X_S)-m}.
	\end{equation}
	Consequently,
	\begin{equation}\label{eq:lifting-limit}
		\norm{M_t}-\norm{A_t+B_t}
		\longrightarrow
		\delta(X_S)
		\qquad (t\to\infty).
	\end{equation}
\end{theorem}

\begin{proof}
	By the Schur triangularization theorem, there is an orthonormal
	basis \(e_1,e_2\) of \(S\) such that
	\begin{equation}\label{eq:SchurXS}
		X_S=
		\begin{pmatrix}
			\alpha&z\\
			0&\beta
		\end{pmatrix}.
	\end{equation}
	In particular,
	\begin{equation}\label{eq:SchurEntries}
		e_1^*Xe_2=z,
		\qquad
		e_2^*Xe_1=0.
	\end{equation}
	By Lemma \ref{lem:ellipse},
	\begin{equation}\label{eq:delta-z}
		\delta(X_S)=|z|.
	\end{equation}
	
	Define
	\[
	D=e_1e_1^*-e_2e_2^*,
	\]
	where \(D\) acts as zero on \(S^\perp\), and set
	\begin{equation}\label{eq:Qt}
		Q_t=
		\begin{pmatrix}
			tD&X\\
			X^*&-tD
		\end{pmatrix}.
	\end{equation}
	Write
	\[
	Q_t=tK+L,
	\]
	where
	\[
	K=
	\begin{pmatrix}
		D&0\\
		0&-D
	\end{pmatrix},
	\qquad
	L=
	\begin{pmatrix}
		0&X\\
		X^*&0
	\end{pmatrix}.
	\]
	Since
	\[
	L^2=
	\begin{pmatrix}
		XX^*&0\\
		0&X^*X
	\end{pmatrix},
	\]
	we have
	\begin{equation}\label{eq:Lnorm}
		\norm{L}=\norm{X}=m.
	\end{equation}
	
	Consider first the \(+1\) eigenspace of \(K\),
	\[
	E_+
	=
	\Span\left\{
	\binom{e_1}{0},
	\binom{0}{e_2}
	\right\}.
	\]
	The compression of \(L\) to \(E_+\) is
	\[
	L_+
	=
	\begin{pmatrix}
		0&z\\
		\overline z&0
	\end{pmatrix},
	\]
	and therefore
	\[
	\lambda_{\max}(L_+)=|z|.
	\]
	Relative to the orthogonal decomposition
	\[
	\mathbb C^{2n}=E_+\oplus E_+^\perp,
	\]
	write
	\[
	Q_t=
	\begin{pmatrix}
		tI_{E_+}+L_+&C_+\\
		C_+^*&R_+
	\end{pmatrix}.
	\]
	On \(E_+^\perp\), the eigenvalues of \(K\) belong to
	\(\{-1,0\}\).  Hence
	\[
	tK|_{E_+^\perp}\leq0.
	\]
	Moreover, \eqref{eq:Lnorm} implies
	\[
	-mI_{2n}\leq L\leq mI_{2n}.
	\]
	It follows that
	\[
	R_+\leq mI_{E_+^\perp},
	\qquad
	\norm{C_+}\leq m.
	\]
	Since \(t>m\), we also have \(t+|z|>m\).  Lemma
	\ref{lem:block-upper}, applied with \(A=L_+\), \(a=|z|\), and
	\(M=m\), therefore gives
	\begin{equation}\label{eq:lmax-estimate}
		t+|z|
		\leq
		\lambda_{\max}(Q_t)
		\leq
		t+|z|+\frac{m^2}{t+|z|-m}.
	\end{equation}
	
	Next consider the \(-1\) eigenspace of \(K\),
	\[
	E_-
	=
	\Span\left\{
	\binom{e_2}{0},
	\binom{0}{e_1}
	\right\}.
	\]
	By \eqref{eq:SchurEntries}, the compression of \(L\) to \(E_-\)
	is zero:
	\[
	L_-
	=
	\begin{pmatrix}
		0&e_2^*Xe_1\\
		e_1^*X^*e_2&0
	\end{pmatrix}
	=
	0.
	\]
	Thus, relative to
	\[
	\mathbb C^{2n}=E_-\oplus E_-^\perp,
	\]
	we may write
	\[
	Q_t=
	\begin{pmatrix}
		-tI_{E_-}&C_-\\
		C_-^*&R_-
	\end{pmatrix}.
	\]
	On \(E_-^\perp\), the eigenvalues of \(K\) belong to
	\(\{0,1\}\), and hence
	\[
	tK|_{E_-^\perp}\geq0.
	\]
	Using again
	\[
	-mI_{2n}\leq L\leq mI_{2n},
	\]
	we obtain
	\[
	R_-\geq -mI_{E_-^\perp},
	\qquad
	\norm{C_-}\leq m.
	\]
	Lemma \ref{lem:block-lower} therefore yields
	\begin{equation}\label{eq:lmin-estimate}
		-t-\frac{m^2}{t-m}
		\leq
		\lambda_{\min}(Q_t)
		\leq
		-t.
	\end{equation}
	
	Set
	\begin{equation}\label{eq:rt}
		r_t=-\lambda_{\min}(Q_t).
	\end{equation}
	By \eqref{eq:lmin-estimate},
	\[
	r_t\geq t.
	\]
	Define
	\begin{equation}\label{eq:AtBt}
		A_t=r_tI_n+tD,
		\qquad
		B_t=r_tI_n-tD.
	\end{equation}
	Since the spectrum of \(D\) is contained in
	\(\{-1,0,1\}\) and \(r_t\geq t\), we have
	\[
	A_t\geq0,
	\qquad
	B_t\geq0.
	\]
	Furthermore,
	\[
	M_t
	=
	\begin{pmatrix}
		A_t&X\\
		X^*&B_t
	\end{pmatrix}
	=
	r_tI_{2n}+Q_t.
	\]
	By the definition of \(r_t\),
	\[
	\lambda_{\min}(M_t)
	=
	r_t+\lambda_{\min}(Q_t)
	=
	0,
	\]
	and hence
	\[
	M_t\geq0.
	\]
	
	Since
	\[
	A_t+B_t=2r_tI_n,
	\]
	we have
	\[
	\norm{A_t+B_t}=2r_t.
	\]
	Also, since \(M_t\geq0\),
	\[
	\norm{M_t}
	=
	\lambda_{\max}(M_t)
	=
	r_t+\lambda_{\max}(Q_t).
	\]
	Using \(r_t=-\lambda_{\min}(Q_t)\), we obtain the exact identity
	\begin{equation}\label{eq:defect-sum}
		\norm{M_t}-\norm{A_t+B_t}
		=
		\lambda_{\max}(Q_t)
		+
		\lambda_{\min}(Q_t).
	\end{equation}
	
	Combining the lower bounds in
	\eqref{eq:lmax-estimate} and \eqref{eq:lmin-estimate} gives
	\[
	\norm{M_t}-\norm{A_t+B_t}
	\geq
	|z|-\frac{m^2}{t-m},
	\]
	while combining the corresponding upper bounds gives
	\[
	\norm{M_t}-\norm{A_t+B_t}
	\leq
	|z|
	+
	\frac{m^2}{t+|z|-m}.
	\]
	In view of \eqref{eq:delta-z}, this proves
	\eqref{eq:quantitative-lifting}.  Since both error terms tend to
	zero as \(t\to\infty\), \eqref{eq:lifting-limit} follows.
\end{proof}
We can now identify the optimal defect exactly.

\begin{theorem}[Exact norm-defect formula]\label{thm:main}
	For every \(X\in M_n(\mathbb C)\), \(n\geq2\),
	\begin{equation}\label{eq:main}
		\Phi(X)=\delta_2(X).
	\end{equation}
	Equivalently,
	\begin{equation}\label{eq:main-expanded}
		\sup
		\left\{
		\norm{
			\begin{pmatrix}
				A&X\\
				X^*&B
		\end{pmatrix}}
		-\norm{A+B}:
		\begin{pmatrix}
			A&X\\
			X^*&B
		\end{pmatrix}\geq0
		\right\}
		=
		\delta_2(X).
	\end{equation}
\end{theorem}

\begin{proof}
	The case \(j=0\) of \eqref{eq:BL-eigenvalue} yields
	\[
	\norm{
		\begin{pmatrix}
			A&X\\
			X^*&B
	\end{pmatrix}}
	-\norm{A+B}
	\leq
	\delta_2(X)
	\]
	for every pair \(A,B\) for which
	\[
	\begin{pmatrix}
		A&X\\
		X^*&B
	\end{pmatrix}\geq0.
	\]
	Taking the supremum over all such pairs gives
	\[
	\Phi(X)\leq\delta_2(X).
	\]
	
	For the reverse inequality, fix any two-dimensional subspace
	\(S\subset\mathbb C^n\).  By Theorem \ref{thm:lifting}, there exists
	for every \(t>\norm{X}\) an admissible positive block matrix
	\[
	M_t=
	\begin{pmatrix}
		A_t&X\\
		X^*&B_t
	\end{pmatrix}
	\]
	such that
	\[
	\norm{M_t}-\norm{A_t+B_t}
	\longrightarrow
	\delta(X_S)
	\qquad (t\to\infty).
	\]
	Since each \(M_t\) is admissible in the definition of \(\Phi(X)\),
	\[
	\Phi(X)
	\geq
	\norm{M_t}-\norm{A_t+B_t}
	\]
	for every \(t>\norm{X}\).  Letting \(t\to\infty\) therefore gives
	\[
	\Phi(X)\geq\delta(X_S).
	\]
	As \(S\) was arbitrary,
	\[
	\Phi(X)
	\geq
	\sup_{\dim S=2}\delta(X_S)
	=
	\delta_2(X).
	\]
	Combining the two inequalities proves \eqref{eq:main}.
\end{proof}
	
\begin{corollary}[Optimality]\label{cor:optimality}
	For a fixed \(X\in M_n(\mathbb C)\), \(n\geq2\), the quantity
	\(\delta_2(X)\) is the smallest constant \(c\geq0\) such that
	\[
	\norm{
		\begin{pmatrix}
			A&X\\
			X^*&B
	\end{pmatrix}}
	\leq
	\norm{A+B}+c
	\]
	for every \(A,B\in M_n(\mathbb C)\) satisfying
	\[
	\begin{pmatrix}
		A&X\\
		X^*&B
	\end{pmatrix}\geq0.
	\]
\end{corollary}

\begin{proof}
	By Theorem \ref{thm:main},
	\[
	\sup
	\left\{
	\norm{
		\begin{pmatrix}
			A&X\\
			X^*&B
	\end{pmatrix}}
	-\norm{A+B}:
	\begin{pmatrix}
		A&X\\
		X^*&B
	\end{pmatrix}\geq0
	\right\}
	=
	\delta_2(X).
	\]
	Hence the stated inequality holds with \(c=\delta_2(X)\).
	Conversely, if it holds with some \(c\geq0\), then taking the
	supremum over all admissible pairs \(A,B\) gives
	\[
	\delta_2(X)=\Phi(X)\leq c.
	\]
	Thus \(\delta_2(X)\) is the smallest possible constant.
\end{proof}

The exact formula also reveals several invariances of the
positive-block defect that are not transparent from its definition.

\begin{corollary}\label{cor:invariances}
	Let \(X\in M_n(\mathbb C)\), \(n\geq2\), let \(U\in M_n(\mathbb C)\)
	be unitary, and let \(\alpha,\beta\in\mathbb C\).  Then
	\[
	\Phi(U^*XU)=\Phi(X),
	\qquad
	\Phi(X^*)=\Phi(X),
	\]
	and
	\begin{equation}\label{eq:affine}
		\Phi(\alpha X+\beta I)
		=
		|\alpha|\,\Phi(X).
	\end{equation}
	In particular,
	\[
	\Phi(X+\beta I)=\Phi(X).
	\]
\end{corollary}

\begin{proof}
	By Theorem \ref{thm:main}, it suffices to verify the corresponding
	identities for \(\delta_2\).
	
	For a unitary \(U\) and a two-dimensional subspace \(S\), the
	compression of \(U^*XU\) to \(S\) is unitarily equivalent to the
	compression of \(X\) to \(US\).  Hence
	\[
	\delta_2(U^*XU)=\delta_2(X).
	\]
	Moreover,
	\[
	(X^*)_S=(X_S)^*,
	\]
	and
	\[
	W((X_S)^*)=\overline{W(X_S)},
	\]
	so
	\[
	\delta_2(X^*)=\delta_2(X).
	\]
	
	Finally, for every two-dimensional subspace \(S\),
	\[
	(\alpha X+\beta I)_S
	=
	\alpha X_S+\beta I_S,
	\]
	and therefore
	\[
	W((\alpha X+\beta I)_S)
	=
	\alpha W(X_S)+\beta.
	\]
	Translation does not change the diameter of the largest inscribed
	disk, while multiplication by \(\alpha\) scales it by
	\(|\alpha|\).  Thus
	\[
	\delta((\alpha X+\beta I)_S)
	=
	|\alpha|\,\delta(X_S).
	\]
	Taking the supremum over all two-dimensional subspaces \(S\) gives
	\[
	\delta_2(\alpha X+\beta I)
	=
	|\alpha|\,\delta_2(X).
	\]
	The assertions now follow from
	\(\Phi=\delta_2\).
\end{proof}
	
\begin{remark}
	The translation invariance
	\[
	\Phi(X+\beta I)=\Phi(X)
	\]
	is not transparent from the extremal definition \eqref{eq:Phi}.
	It becomes immediate once one identifies
	\[
	\Phi(X)=\delta_2(X),
	\]
	since translating a numerical range does not change the diameter of
	its largest inscribed disk.
\end{remark}

\begin{remark}
	Theorem \ref{thm:main} has a local-to-global character.  The
	quantity \(\delta_2(X)\) is determined entirely by the
	two-dimensional compressions of \(X\), while \(\Phi(X)\) is defined
	through an optimization over positive block matrices of order
	\(2n\).  Theorem \ref{thm:lifting} supplies the bridge between these
	two descriptions by asymptotically realizing the contribution of
	each two-dimensional compression through a family of admissible
	positive block matrices.
\end{remark}
	
\section{The zero-width case and the Bourin--Lee conjecture}

We next identify the zero set of \(\delta_2\).  The following
finite-dimensional characterization is closely related to classical
results concerning normal two-dimensional compressions; compare
Farenick \cite{Farenick}.  We include a direct proof in the precise
form needed here.

\begin{proposition}\label{prop:zero}
	For \(X\in M_n(\mathbb C)\), \(n\geq2\), the following are
	equivalent:
	\begin{enumerate}[label=\textup{(\roman*)}]
		\item \(\delta_2(X)=0\);
		\item every two-dimensional compression of \(X\) is normal;
		\item \(X\) is essentially Hermitian.
	\end{enumerate}
\end{proposition}

\begin{proof}
	By the definition of \(\delta_2\),
	\[
	\delta_2(X)=0
	\]
	if and only if
	\[
	\delta(X_S)=0
	\]
	for every two-dimensional subspace \(S\subset\mathbb C^n\).
	By Lemma \ref{lem:ellipse}, a \(2\times2\) matrix has zero
	minor-axis length if and only if it is normal.  This proves the
	equivalence of \textup{(i)} and \textup{(ii)}.
	
	We next prove that \textup{(ii)} implies \textup{(iii)}.  Write
	\[
	X=H+iK,
	\qquad
	H=\frac{X+X^*}{2},
	\qquad
	K=\frac{X-X^*}{2i},
	\]
	where \(H\) and \(K\) are Hermitian.
	
	For every two-dimensional subspace \(S\),
	\[
	X_S=H_S+iK_S.
	\]
	Since \(X_S\) is normal,
	\[
	0=X_S^*X_S-X_SX_S^*
	=2i[H_S,K_S],
	\]
	and therefore
	\begin{equation}\label{eq:commute}
		[H_S,K_S]=0
		\qquad
		\text{for every }S\subset\mathbb C^n,\ \dim S=2.
	\end{equation}
	
	If \(H=cI\) is scalar, then
	\[
	X=cI+iK,
	\]
	so \(X\) is essentially Hermitian.  We may therefore assume that
	\(H\) is not scalar.
	
	Choose an orthonormal eigenbasis \(e_1,\ldots,e_n\) for \(H\), with
	\[
	He_j=h_je_j,
	\qquad
	h_j\in\mathbb R,
	\]
	and write
	\[
	k_{ij}=e_i^*Ke_j.
	\]
	
	Suppose first that \(h_i\neq h_j\).  For
	\[
	S=\Span\{e_i,e_j\},
	\]
	we have
	\[
	H_S=
	\begin{pmatrix}
		h_i&0\\
		0&h_j
	\end{pmatrix}.
	\]
	Since \(H_S\) commutes with \(K_S\) and \(h_i\neq h_j\), the
	off-diagonal entries of \(K_S\) vanish.  Hence
	\begin{equation}\label{eq:crosszero}
		k_{ij}=0.
	\end{equation}
	
	Now let \(i\neq j\) satisfy
	\[
	h_i=h_j=\lambda.
	\]
	Since \(H\) is not scalar, there exists an index \(\ell\) such that
	\[
	h_\ell=\mu\neq\lambda.
	\]
	Set
	\[
	u=\frac{e_i+e_\ell}{\sqrt2},
	\qquad
	v=e_j.
	\]
	Then \(u,v\) are orthonormal, and the compression of \(H\) to
	\[
	S=\Span\{u,v\}
	\]
	is
	\[
	H_S=
	\begin{pmatrix}
		(\lambda+\mu)/2&0\\
		0&\lambda
	\end{pmatrix}.
	\]
	Its diagonal entries are distinct.  Since
	\([H_S,K_S]=0\), it follows that the off-diagonal entries of
	\(K_S\) vanish, and hence
	\[
	u^*Kv=0.
	\]
	On the other hand,
	\[
	u^*Kv
	=
	\frac{1}{\sqrt2}
	\left(k_{ij}+k_{\ell j}\right).
	\]
	Since \(h_\ell\neq h_j\), equation \eqref{eq:crosszero} gives
	\(k_{\ell j}=0\).  Thus
	\[
	k_{ij}=0.
	\]
	
	We have therefore shown that \(K\) is diagonal in the same basis as
	\(H\).  Write
	\[
	H=\diag(h_1,\ldots,h_n),
	\qquad
	K=\diag(k_1,\ldots,k_n),
	\]
	where \(k_j=k_{jj}\in\mathbb R\).
	
	If \(n=2\), then \(H\) is non-scalar, so \(h_1\neq h_2\).  Hence
	there exist real numbers \(a,b\) such that
	\[
	k_j=ah_j+b,
	\qquad j=1,2.
	\]
	Therefore
	\[
	K=aH+bI.
	\]
	
	Assume now that \(n\geq3\).  Let \(i,j,\ell\) be three distinct
	indices and set
	\[
	u=\frac{e_i+e_j+e_\ell}{\sqrt3},
	\qquad
	v=\frac{e_i-e_j}{\sqrt2}.
	\]
	Then \(u,v\) are orthonormal.  Relative to this basis, the
	compressions of \(H\) and \(K\) are
	\[
	H_S=
	\begin{pmatrix}
		\dfrac{h_i+h_j+h_\ell}{3}
		&
		\dfrac{h_i-h_j}{\sqrt6}
		\\[6pt]
		\dfrac{h_i-h_j}{\sqrt6}
		&
		\dfrac{h_i+h_j}{2}
	\end{pmatrix}
	\]
	and
	\[
	K_S=
	\begin{pmatrix}
		\dfrac{k_i+k_j+k_\ell}{3}
		&
		\dfrac{k_i-k_j}{\sqrt6}
		\\[6pt]
		\dfrac{k_i-k_j}{\sqrt6}
		&
		\dfrac{k_i+k_j}{2}
	\end{pmatrix}.
	\]
	By \eqref{eq:commute}, these matrices commute.  Equating the
	\((1,2)\)-entry of their commutator to zero gives
	\[
	h_i k_j-h_i k_\ell-h_j k_i+h_j k_\ell
	+h_\ell k_i-h_\ell k_j=0.
	\]
	Equivalently,
	\begin{equation}\label{eq:detzero}
		\det
		\begin{pmatrix}
			1&h_i&k_i\\
			1&h_j&k_j\\
			1&h_\ell&k_\ell
		\end{pmatrix}
		=0.
	\end{equation}
	Thus every three of the points
	\[
	(h_j,k_j)\in\mathbb R^2
	\]
	are collinear.
	
	Since \(H\) is not scalar, choose \(p,q\) such that
	\[
	h_p\neq h_q.
	\]
	Define
	\[
	a=\frac{k_p-k_q}{h_p-h_q},
	\qquad
	b=k_p-ah_p.
	\]
	For any \(r\notin\{p,q\}\), applying \eqref{eq:detzero} to the
	triple \(p,q,r\) gives
	\[
	k_r=ah_r+b.
	\]
	The same identity is already true for \(r=p,q\).  Hence
	\[
	K=aH+bI.
	\]
	Therefore
	\[
	X=H+iK=(1+ia)H+ibI,
	\]
	so \(X\) is essentially Hermitian.
	
	Finally, suppose that \textup{(iii)} holds.  Thus
	\[
	X=\alpha H+\beta I,
	\qquad
	H=H^*.
	\]
	For every two-dimensional subspace \(S\),
	\[
	X_S=\alpha H_S+\beta I_S.
	\]
	Since \(H_S\) is Hermitian, \(X_S\) is normal.  Hence
	\textup{(iii)} implies \textup{(ii)}, completing the proof.
\end{proof}
We now obtain the announced resolution of the conjecture.

\begin{theorem}[Resolution of the Bourin--Lee conjecture]
	\label{thm:conjecture}
	Let \(X\in M_n(\mathbb C)\), \(n\geq2\).  The following conditions
	are equivalent:
	\begin{enumerate}[label=\textup{(\roman*)}]
		\item
		\[
		\norm{
			\begin{pmatrix}
				A&X\\
				X^*&B
		\end{pmatrix}}
		\leq
		\norm{A+B}
		\]
		for every \(A,B\in M_n(\mathbb C)\) such that
		\[
		\begin{pmatrix}
			A&X\\
			X^*&B
		\end{pmatrix}\geq0;
		\]
		
		\item \(\Phi(X)=0\);
		
		\item \(\delta_2(X)=0\);
		
		\item \(X\) is essentially Hermitian.
	\end{enumerate}
	In particular, Conjecture 3.3 of Bourin and Lee
	\cite{BourinLeeInradius} is true.
\end{theorem}

\begin{proof}
	By Theorem \ref{thm:main},
	\[
	\Phi(X)=\delta_2(X),
	\]
	so \textup{(ii)} and \textup{(iii)} are equivalent.  Proposition
	\ref{prop:zero} gives the equivalence of \textup{(iii)} and
	\textup{(iv)}.
	
	It remains to relate \textup{(i)} and \textup{(ii)}.  If
	\textup{(i)} holds, then every quantity appearing in the supremum
	defining \(\Phi(X)\) is nonpositive.  Hence
	\[
	\Phi(X)\leq0.
	\]
	On the other hand, Theorem \ref{thm:main} gives
	\[
	\Phi(X)=\delta_2(X)\geq0.
	\]
	Therefore
	\[
	\Phi(X)=0,
	\]
	and \textup{(ii)} follows.
	
	Conversely, suppose that \(\Phi(X)=0\).  For every admissible pair
	\(A,B\),
	\[
	\norm{
		\begin{pmatrix}
			A&X\\
			X^*&B
	\end{pmatrix}}
	-\norm{A+B}
	\leq
	\Phi(X)=0.
	\]
	This is exactly \textup{(i)}, and the proof is complete.
\end{proof}

\begin{remark}
	Theorem \ref{thm:conjecture} shows that the Bourin--Lee conjecture
	is precisely the zero-defect case of the stronger identity
	\[
	\Phi(X)=\delta_2(X).
	\]
	Thus \(\delta_2(X)\) gives the exact maximal excess of
	\[
	\norm{
		\begin{pmatrix}
			A&X\\
			X^*&B
	\end{pmatrix}}
	\]
	over \(\norm{A+B}\), among all positive block matrices with
	prescribed off-diagonal block \(X\).  In particular, this excess
	vanishes identically exactly when \(X\) is essentially Hermitian.
\end{remark}
	
\section{Normal off-diagonal blocks and sharpness}

We finally address the sharpness problem posed in
\cite[Question 2.5]{BourinLeeInradius}. We first construct an explicit
normal example proving sharpness at the leading eigenvalue level
\(j=0\). We then show that a direct-sum amplification of the same
construction yields sharpness at every prescribed eigenvalue level
\(j\geq0\).
	
	Let
	\[
	\zeta=e^{2\pi i/3},
	\qquad
	U=\diag(1,\zeta,\zeta^2),
	\]
	and define
	\[
	q_0=\frac1{\sqrt3}
	\begin{pmatrix}
		1\\1\\1
	\end{pmatrix},
	\qquad
	q_1=Uq_0,
	\qquad
	q_2=U^2q_0.
	\]
	Since
	\[
	1+\zeta+\zeta^2=0,
	\]
	the vectors \(q_0,q_1,q_2\) form an orthonormal basis of
	\(\mathbb C^3\).  Moreover,
	\begin{equation}\label{eq:cycle}
		Uq_0=q_1,
		\qquad
		Uq_1=q_2,
		\qquad
		Uq_2=q_0.
	\end{equation}
	
	\begin{proposition}\label{prop:explicit}
		Let
		\[
		D=q_1q_1^*-q_0q_0^*.
		\]
		For \(t>0\), set
		\[
		s_t=\sqrt{t^2+4},
		\qquad
		k_t=t+s_t,
		\]
		and define
		\[
		A_t=\frac{k_t}{2}I_3+tD,
		\qquad
		B_t=\frac{k_t}{2}I_3-tD.
		\]
		Then \(A_t,B_t\geq0\) and
		\[
		M_t=
		\begin{pmatrix}
			A_t&U\\
			U^*&B_t
		\end{pmatrix}\geq0.
		\]
		Moreover,
		\begin{equation}\label{eq:defect-explicit}
			\norm{M_t}-\norm{A_t+B_t}
			=
			1-\frac{2}{t+\sqrt{t^2+4}}.
		\end{equation}
		In particular,
		\[
		\norm{M_t}-\norm{A_t+B_t}
		\longrightarrow1
		\qquad (t\to\infty).
		\]
	\end{proposition}
	
	\begin{proof}
		Since
		\[
		Dq_1=q_1,\qquad
		Dq_0=-q_0,\qquad
		Dq_2=0,
		\]
		the spectrum of \(D\) is \(\{-1,0,1\}\).  Since
		\[
		\frac{k_t}{2}
		=
		\frac{t+\sqrt{t^2+4}}{2}
		>t,
		\]
		it follows immediately that
		\[
		A_t\geq0,
		\qquad
		B_t\geq0.
		\]
		
		Write
		\[
		M_t=\frac{k_t}{2}I_6+Q_t,
		\qquad
		Q_t=
		\begin{pmatrix}
			tD&U\\
			U^*&-tD
		\end{pmatrix}.
		\]
		Using \eqref{eq:cycle}, together with
		\[
		U^*q_0=q_2,\qquad
		U^*q_1=q_0,\qquad
		U^*q_2=q_1,
		\]
		we see that \(\mathbb C^3\oplus\mathbb C^3\) is the orthogonal
		direct sum of the following three \(Q_t\)-invariant subspaces:
		\[
		\Span\left\{
		\binom{q_1}{0},\binom{0}{q_0}
		\right\},
		\qquad
		\Span\left\{
		\binom{q_2}{0},\binom{0}{q_1}
		\right\},
		\qquad
		\Span\left\{
		\binom{q_0}{0},\binom{0}{q_2}
		\right\}.
		\]
		Relative to the displayed ordered bases, the corresponding
		restrictions of \(Q_t\) are
		\[
		\begin{pmatrix}
			t&1\\
			1&t
		\end{pmatrix},
		\qquad
		\begin{pmatrix}
			0&1\\
			1&-t
		\end{pmatrix},
		\qquad
		\begin{pmatrix}
			-t&1\\
			1&0
		\end{pmatrix}.
		\]
		The first matrix has eigenvalues
		\[
		t+1,\qquad t-1,
		\]
		while each of the other two has eigenvalues
		\[
		\frac{-t\pm\sqrt{t^2+4}}{2}.
		\]
		Consequently,
		\[
		\lambda_{\max}(Q_t)=t+1
		\]
		and
		\[
		\lambda_{\min}(Q_t)
		=
		-\frac{t+\sqrt{t^2+4}}{2}
		=
		-\frac{k_t}{2}.
		\]
		Therefore
		\[
		\lambda_{\min}(M_t)
		=
		\frac{k_t}{2}+\lambda_{\min}(Q_t)
		=
		0,
		\]
		and hence
		\[
		M_t\geq0.
		\]
		
		Furthermore,
		\[
		A_t+B_t=k_tI_3,
		\]
		so
		\[
		\norm{A_t+B_t}=k_t.
		\]
		Since \(M_t\geq0\),
		\[
		\norm{M_t}
		=
		\lambda_{\max}(M_t)
		=
		\frac{k_t}{2}+t+1.
		\]
		Thus
		\begin{align*}
			\norm{M_t}-\norm{A_t+B_t}
			&=
			t+1-\frac{k_t}{2}\\
			&=
			1-\frac{\sqrt{t^2+4}-t}{2}\\
			&=
			1-\frac{2}{t+\sqrt{t^2+4}},
		\end{align*}
		which proves \eqref{eq:defect-explicit}.
	\end{proof}
	
	The preceding \(3\times3\) construction shows that the Bourin--Lee
	radius correction is sharp at the leading eigenvalue level \(j=0\).
	A natural question is whether this phenomenon is confined to the
	largest eigenvalue or persists throughout the full family of
	eigenvalue inequalities.  The key observation is that the basic
	extremal example can be amplified by direct sums: this preserves
	positivity and normality, keeps the partial trace scalar,
	and increases the multiplicity of the top eigenvalue to any prescribed
	level.  This allows the same asymptotic extremality to be transferred
	from \(j=0\) to every \(j\geq0\).
	
\begin{theorem}[Sharpness of the normal radius constant at every eigenvalue level]
	\label{thm:radius}
	Let \(j\geq0\) be an integer. For every \(r>0\) and every
	\(\varepsilon>0\), there exist an integer \(n\), a normal matrix
	\(N\in M_n(\mathbb C)\), and matrices \(A,B\in M_n(\mathbb C)\) such
	that
	\[
	\spec(N)\subset\{z\in\mathbb C:|z|\leq r\},
	\]
	\[
	\begin{pmatrix}
		A&N\\
		N^*&B
	\end{pmatrix}\geq0,
	\]
	and
	\begin{equation}\label{eq:radius-sharp}
		\lambda_{1+2j}
		\begin{pmatrix}
			A&N\\
			N^*&B
		\end{pmatrix}
		>
		\lambda_{1+j}(A+B)+r-\varepsilon.
	\end{equation}
	Consequently, the constant \(r\) in the Bourin--Lee eigenvalue
	estimate is best possible at every prescribed eigenvalue level
	\(j\geq0\). In particular, this gives an affirmative answer to
	\cite[Question 2.5]{BourinLeeInradius}.
\end{theorem}

\begin{proof}
	Fix \(j\geq0\) and set
	\[
	q=2j+1,
	\qquad
	n=3q.
	\]
	Let \(U\), \(A_t\), \(B_t\), and \(M_t\) be as in Proposition
	\ref{prop:explicit}. Define
	\[
	\widetilde A_t=A_t^{\oplus q},
	\qquad
	\widetilde B_t=B_t^{\oplus q},
	\qquad
	\widetilde N=U^{\oplus q},
	\]
	and set
	\[
	A=r\widetilde A_t,
	\qquad
	B=r\widetilde B_t,
	\qquad
	N=r\widetilde N.
	\]
	Since \(U\) is unitary, \(N\) is normal and
	\[
	\spec(N)
	=
	\{r,r\zeta,r\zeta^2\}
	\subset
	\{z\in\mathbb C:|z|\leq r\}.
	\]
	
	Moreover,
	\[
	\begin{pmatrix}
		A&N\\
		N^*&B
	\end{pmatrix}
	=
	r
	\begin{pmatrix}
		\widetilde A_t&\widetilde N\\
		\widetilde N^*&\widetilde B_t
	\end{pmatrix}.
	\]
	After a simultaneous permutation of rows and columns, the matrix on
	the right is the direct sum of \(q\) copies of \(M_t\). Hence
	\[
	\begin{pmatrix}
		A&N\\
		N^*&B
	\end{pmatrix}
	\geq0.
	\]
	
	By Proposition \ref{prop:explicit},
	\[
	A_t+B_t=k_tI_3,
	\]
	and therefore
	\[
	A+B
	=
	rk_tI_n.
	\]
	Consequently,
	\begin{equation}\label{eq:partial-trace-eigenvalue}
		\lambda_{1+j}(A+B)=rk_t.
	\end{equation}
	
	On the other hand, the largest eigenvalue of each copy of \(rM_t\)
	is
	\[
	r\lambda_1(M_t),
	\]
	so the largest eigenvalue of the direct sum occurs with multiplicity
	at least \(q\). Since
	\[
	q=2j+1,
	\]
	it follows that
	\begin{equation}\label{eq:amplified-eigenvalue}
		\lambda_{1+2j}
		\begin{pmatrix}
			A&N\\
			N^*&B
		\end{pmatrix}
		=
		r\lambda_1(M_t).
	\end{equation}
	Combining \eqref{eq:partial-trace-eigenvalue},
	\eqref{eq:amplified-eigenvalue}, and Proposition
	\ref{prop:explicit}, we obtain
	\begin{align*}
		\lambda_{1+2j}
		\begin{pmatrix}
			A&N\\
			N^*&B
		\end{pmatrix}
		-\lambda_{1+j}(A+B)
		&=
		r\bigl(\lambda_1(M_t)-k_t\bigr)\\
		&=
		r\left(
		1-\frac{2}{t+\sqrt{t^2+4}}
		\right).
	\end{align*}
	The right-hand side tends to \(r\) as \(t\to\infty\). Hence, for
	sufficiently large \(t\),
	\[
	\lambda_{1+2j}
	\begin{pmatrix}
		A&N\\
		N^*&B
	\end{pmatrix}
	-\lambda_{1+j}(A+B)
	>
	r-\varepsilon,
	\]
	which is precisely \eqref{eq:radius-sharp}.
\end{proof}

\begin{remark}
	The sharpness phenomenon already occurs for a \(3\times3\) normal
	off-diagonal block.  This is the smallest possible dimension for
	genuinely two-dimensional numerical-range geometry of a normal
	matrix: the numerical range of a normal \(2\times2\) matrix is the
	convex hull of at most two eigenvalues and hence is a line segment,
	whereas the matrix \(U\) above has three noncollinear eigenvalues.
\end{remark}
	
\section{Concluding remarks}

The identity
\[
\Phi(X)=\delta_2(X)
\]
gives an exact extremal interpretation of the elliptical width.
Although \(\delta_2(X)\) is defined entirely through the
two-dimensional compressions of \(X\), the quantity \(\Phi(X)\) is
defined by a global optimization over positive block matrices of order
\(2n\). The compression-lifting theorem shows that these two
apparently different descriptions coincide exactly.

The proof also reveals the mechanism behind this identity. For a
two-dimensional Schur compression
\[
X_S=
\begin{pmatrix}
	\alpha&z\\
	0&\beta
\end{pmatrix},
\]
the quantity
\[
|z|=\delta(X_S)
\]
appears at the upper spectral edge of a suitable block perturbation,
whereas the vanishing reverse Schur entry eliminates the corresponding
leading-order contribution at the lower spectral edge. After shifting
by the lower spectral edge to restore positivity, the difference
between the two extreme spectral contributions becomes precisely the
positive-block norm defect. The estimates in Theorem
\ref{thm:lifting} quantify this construction with an error of order
\(t^{-1}\).

The Bourin--Lee conjecture is therefore the zero-defect case of the
same extremal principle:
\[
X\ \text{is essentially Hermitian}
\quad\Longleftrightarrow\quad
\delta_2(X)=0
\quad\Longleftrightarrow\quad
\Phi(X)=0.
\]
Thus the qualitative rigidity conjectured in
\cite{BourinLeeInradius} is governed by an exact quantitative
invariant.

The normal sharpness problem exhibits a complementary phenomenon.
The explicit \(3\times3\) construction in Section 5 shows that the
radius correction in the Bourin--Lee estimate is asymptotically sharp
at the leading eigenvalue level. Direct-sum amplification transfers the
same extremal behavior to every prescribed level \(j\geq0\). Hence the
radius constant in the normal-off-diagonal eigenvalue inequality is
best possible at every eigenvalue level.

A different extremal problem remains when the off-diagonal block
\(X\in M_n(\mathbb C)\) is fixed. For \(j=1,\ldots,n-1\), define
\[
\Phi_j(X)
=
\sup
\left\{
\lambda_{1+2j}
\begin{pmatrix}
	A&X\\
	X^*&B
\end{pmatrix}
-\lambda_{1+j}(A+B):
\begin{pmatrix}
	A&X\\
	X^*&B
\end{pmatrix}\geq0
\right\}.
\]
The Bourin--Lee inequality implies
\[
\Phi_j(X)\leq\delta_2(X).
\]
It remains to determine whether equality holds for every fixed
\(X\), or whether the optimal higher-eigenvalue defect is governed by
a different invariant involving higher-rank information or several
compression directions simultaneously. The lifting method developed
here is intrinsically based on a two-dimensional Schur compression
and on the interaction of the two extreme spectral edges, suggesting
that the fixed-\(X\) higher-eigenvalue problem may require a genuinely
higher-rank analogue of the present construction.

\section*{Declarations}

\noindent\textbf{Conflict of interest.}
The author declares that there is no conflict of interest.

\medskip

\noindent\textbf{Funding.}
The author received no specific funding for this work.

\medskip

\noindent\textbf{Data availability.}
No datasets were generated or analyzed during the current study.

\medskip

\noindent\textbf{Author contributions.}
The author was responsible for the conception of the work, the development
of the mathematical results, the proofs, and the preparation of the
manuscript.

\end{document}